\documentclass[a4paper, 12pt]{article}

\usepackage[latin1]{inputenc} 
\usepackage[english]{babel} 

\usepackage{amsmath,amssymb,amsthm,mathrsfs}

\makeatletter
\def\thm@space@setup{%
  \thm@preskip=\parskip \thm@postskip=0pt
}
\makeatother 

\usepackage{enumerate}

\usepackage{longtable}
\usepackage{hhline}

\usepackage{graphicx}

\usepackage[usenames]{color}

\usepackage{indentfirst}

\usepackage[all]{xy}

\usepackage{makeidx}
\makeindex

\usepackage{geometry}
\DeclareMathOperator{\fwc}{fwc}
\DeclareMathOperator{\inte}{int}

\newtheorem{teo}{Theorem}[section] 
\newtheorem{cor}[teo]{Corolary} 
\newtheorem{defi}[teo]{Definition} 
\newtheorem{exem}[teo]{Exemple} 
\newtheorem{lema}[teo]{Lemma} 
\newtheorem{prop}[teo]{Proposition} 

\theoremstyle{remark} 

\title{A metric for the set of all additive basis}
\author{Luan Alberto Ferreira} 
\date{\today} 

\newcommand{\N}{\mathbb{N}}
\newcommand{\Z}{\mathbb{Z}}

\newcommand{\A}{\mathbb{A}}

\newcommand{\PP}{\mathbb{P}}

\newcommand{\B}{\mathcal{B}}

\newcommand{\pseudo}{\stackrel{p}{\rightarrow}}

\newcommand{\dibigcup}{\displaystyle \bigcup}
\newcommand{\dibigcap}{\displaystyle \bigcap}
\newcommand{\disum}{\displaystyle \sum}

\newcommand{\difrac}{\displaystyle \frac}
\newcommand{\dilim}{\displaystyle \lim}

\newcommand{\fantasma}[3]{{#1}_{{#2}}^{\phantom{{#2}}{#3}}}

\newcommand{\soma}[2]{\disum_{\alpha \in {#1} \triangle {#2}} \difrac{1}{\alpha^2}}
\newcommand{\dd}[2]{|o({#1}) - o({#2})| + \soma{#1}{#2}}

\renewenvironment{proof}{\noindent \textit{Proof:}}{\qed}

\begin{document}

\maketitle

\abstract{The aim of this article is to present a topological tool for the study of additive basis in additive number theory. It will be proposal a metric for the set of all additive basis, in which it will be possible to study properties of some additive bases studying basis near the chosen basis. This metric allow, for example, detect some additive basis in which we can add some integers in it without change its order.}

\tableofcontents

\section{Introduction}

An important problem in additive number theory is to decide if a subset $A \subseteq \N_0$ is an additive basis and, if so, determine its order. For example, we know that $\PP = \{0, 1, 2, 3, 5, 7, 11, 13, 17, 19, 23, \ldots\}$ is an additive basis, but we don't know exactly its order.

Another problem yet in this area is the following: suppose its given an additive basis $A$. Is it true that $o(A) = o(A \cup \Gamma), \ \forall \ \Gamma \subseteq \N$ finite?

This article will present a topological tool for the study of this kind of problems, in terms of a metric for the set of all additive basis.

For the basic results in additive number theory, we recommend Nathanson. For the basic results in metric spaces, see Engelking. We will use this results freely, without any further comment.

\section{Notation}

\begin{itemize}
  \item $\N_0 = \{z \in \Z; \ z \geq 0\}$
  \item $\N = \{z \in \Z; \ z \geq 1\}$
  \item If $\alpha \in \N$, denote $\overleftarrow{n} = \{\beta \in \N_0; \ 0 \leq \beta \leq \alpha\}$.
\end{itemize}

\section{Metric in $\A$}

Denote by $\A$ the set of all additive basis.

\begin{teo}

The function $$\begin{array}[t]{cccc} d: & \A \times \A & \rightarrow & [0, \infty) \\ & (A,B) & \mapsto & |o(A) - o(B)| + \soma{A}{B} \end{array}$$ is a metric in $\A$.

\end{teo}

\begin{proof} We only need to prove in triangular inequality. For this, let $A, B, C \in \A$. As $$A \triangle C \subseteq (A \triangle B) \cup (B \triangle C),$$ then \begin{eqnarray*}d(A, C) & = & \dd{A}{C} \leq \\ & \leq & \dd{A}{B} + \dd{B}{C} = \\ & = & d(A, B) + d(B, C).\end{eqnarray*} \end{proof}

Note that the function $\begin{array}[t]{cccc} \widehat{d}: & \A \times \A & \rightarrow & [0, \infty) \\ & (A,B) & \mapsto & |o(A) - o(B)| \end{array}$ is a pseudometric.

\begin{exem}

Let's calculate the distance between two additive basis very known. By Lagrange's theorem, $o(\N^2) = 4$. So: \begin{eqnarray*} d(\N, \N^2) & = & \dd{\N}{\N^2} = \\ & = & |1 - 4| + \disum_{\alpha \in \N - \N^2} \difrac{1}{\alpha^2} = \\ & = & 3 + \disum_{\alpha = 2}^{\infty} \difrac{1}{\alpha^2} - \disum_{\alpha = 2}^{\infty} \difrac{1}{\alpha^4} = \\ & = & 3 + \left(\difrac{\pi^2}{6} - 1\right) - \left(\difrac{\pi^4}{90} - 1\right) = \\ & = & 3 + \difrac{\pi^2}{6} - \difrac{\pi^4}{90}. \end{eqnarray*}

\end{exem}

\begin{lema} If $d(A, B) \in \N$, then $A = B$. Moreover, if $x \in \N^*$ and $d(A, B) < \difrac{1}{x^2}$, then $A \cap \overleftarrow{x} = B \cap \overleftarrow{x}$.

\end{lema}

\begin{proof} If $d(A, B) \in \N$, then $\soma{A}{B} = 0$, once $|o(A) - o(B)| \in \N$ and $0 \leq \soma{A}{B} \leq \disum_{\alpha = 2}^{\infty} \difrac{1}{\alpha^2} = \difrac{\pi^2}{6} - 1 < 1$. So $A = B$. The other claim is proved by contra-reciprocal: if $A \cap \overleftarrow{x} \neq B \cap \overleftarrow{x}$, then exists $y \in (A \triangle B) \cap \overleftarrow{x} \Rightarrow d(A, B) \geq \difrac{1}{y^2} \geq \difrac{1}{x^2}$. \end{proof}

\begin{prop}

The function $\begin{array}[t]{cccc} o: & \A & \rightarrow & \N^* \\ & A & \mapsto & o(A) \end{array}$ is a weak contraction, but it isn't a contraction.

\end{prop}

\begin{proof} The first affirmation follows directly from the definition of the metric $d$. For the second, suppose by contradiction that exists $c \in [0, 1)$ such that ${|o(A) - o(B)| \leq c \cdot d(A, B), \ \forall \ A, \ B \in \A}$. Clearly $c \neq 0$; in other case all bases would have the same order. So $0 < c < 1$. Let $x \in \N - \{0, 1\}$ such that $c \cdot \left(1 + \difrac{1}{x^2}\right) < 1$. If $A = \N$ e $B = \N - \{x\}$, then $1 = |o(A) - o(B)| \leq c \cdot d(A, B) = c \cdot \left(1 + \difrac{1}{x^2}\right) < 1$, contradiction. \end{proof}

\begin{lema} \label{dibigcapdibigcup}

Let $\{A_n\}_{n \in \N}$ a sequence in $\A$. Suppose that exists $A \subseteq \N$ with the following property: given $\varepsilon > 0$ exists $n_0 \in \N$ such that if $n \geq n_0$, then $\soma{A_n}{A} < \varepsilon$. So $A = \dibigcup_{m = 0}^{\infty} \dibigcap_{n = m}^{\infty} A_n$.

\end{lema}

\begin{proof} Let $B = \dibigcup_{m = 0}^{\infty} \dibigcap_{n = m}^{\infty} A_n$. We will show that $A \subseteq B$ and $A \supseteq B$.

$(\subseteq)$ Let $x \in A$. Like any addictive base contains $\{0, 1\}$, then $\{0, 1\} \subseteq B$. Thereby, we can assume without loss of generality that $x \geq 2$. By hypothesis, given $\varepsilon = \difrac{1}{x^2}$ exists $n_0 \in \N$ such that if $n \geq n_0$, then $\soma{A_n}{A} < \difrac{1}{x^2}$. In particular, $A \cap \overleftarrow{x} = A_n \cap \overleftarrow{x}, \ \forall \ n \geq n_0$. So $x \in B$.

$(\supseteq)$ Let $x \in B, \ x \geq 2$. Then exists $m \in \N$ such that if $n \geq m$, then $x \in A_n$. By hypothesis, given $\varepsilon = \difrac{1}{x^2}$ exists $n_0 \in \N$ such that if $n \geq n_0$, then $\soma{A_n}{A} < \difrac{1}{x^2}$. As before, $A \cap \overleftarrow{x} = A_n \cap \overleftarrow{x}, \ \forall \ n \geq n_0$. If $k = \max\{m, n_0\}$, then $x \in A_k \Rightarrow x \in A$. \end{proof}

\begin{defi}

If $h \in \N^*$, we define $\A_h = \{A \in \A; \ o(A) = h\}$.

\end{defi}

\begin{prop} \label{proposicaodosan}

The sets $\A_h$ satisfy the following properties:

\begin{enumerate}
  \item $\A_h$ is open and closed in $\A$.
  \item $\A_h$ is totally bounded.
  \item $\partial(\A_h) = \emptyset$.
  \item $\A_h$ is complete if, and only if, $h = 1$.
\end{enumerate}

\end{prop}

\begin{proof} \textit{1.} As the function $\begin{array}[t]{cccc} o: & \A & \rightarrow & \N^* \\ & A & \mapsto & o(A) \end{array}$ is continuous in $\A$ e $\A_h = o^{-1}(\{h\})$, follows that $\A_h$ is open and closed in $\A$.

\textit{2.} Let $\varepsilon > 0$ and let $m \in \N - \{0, 1\}$ such that $\disum_{\alpha > m} \difrac{1}{\alpha^2} < \varepsilon$. For each $\Gamma \subseteq \{2, \ldots, m\}$ choose, if possible, $A_\Gamma \in \A_h$ such that $A_\Gamma \cap \{2, \ldots, m\} = \Gamma$. Let $A \in \A_h$. So $$A \cap \{2, \ldots, m\} = A_\Gamma \cap \{2, \ldots, m\},$$ for some $\Gamma \subseteq \{2, \ldots, m\}$. So, $d(A, A_\Gamma) = \soma{A}{A_\Gamma} \leq \disum_{\alpha > m} \difrac{1}{\alpha^2} < \varepsilon$. This implies that $\A_h$ is totally bounded.

\textit{3.} Suppose that exists $A \in \partial(\A_h)$. Then given $r = \difrac{1}{2}$ exists $B_1 \in \A_h$ and $B_2 \in \fantasma{\A}{h}{C}$ such that $B_1, B_2 \in \B(A, r)$. So $d(B_1, A) < \difrac{1}{2} \Rightarrow o(B_1) = o(A)$ and $d(B_2, A) < \difrac{1}{2} \Rightarrow o(B_2) = o(A)$, contradiction.

\textit{4.} Clearly $\A_1$ is complete, because $\A_1 = \{\N\}$. Now, for each $h \geq 2$, we will build a Cauchy sequence $\{A_n\}_{n \in \N^*} \subseteq \A_h$ such that $\{A_n\}_{n \in \N^*}$ does not converge in $\A_h$. Fixed $h \geq 2$, consider the sequence $\{A_n\}_{n \in \N^*}$ given by $$A_n = \{0, 1, \ldots, n, hn + 1, \ldots\}.$$ So $o(A_n) = h, \ \forall \ n \in \N^*$. If $\{A_n\}_{n \in \N^*}$ converges in $\A_h$, it would exists $A \in \A_h$ such that $A_n \rightarrow A$. By lemma \ref{dibigcapdibigcup}, $A = \dibigcup_{m = 1}^{\infty} \dibigcap_{n = m}^{\infty} A_n = \N \Rightarrow A_n \rightarrow \N$, contradiction, because $o(A_n) = h \geq 2, \ \forall \ n \in \N^*$. \end{proof}

From the previous proposition we get two corollaries:

\begin{cor} \label{cortotalmentelimitado}

The following statements are equivalent with respect to a set $X \subseteq \A$:

\begin{enumerate}
  \item $X$ is totally bounded.
  \item $X$ é limited.
  \item Exists $M > 0$ such that $o(A) \leq M, \ \forall \ A \in X$.
\end{enumerate}

\end{cor}

\begin{proof} $(\textit{1.} \Rightarrow \textit{2.})$ Trivial.

$(\textit{2.} \Rightarrow \textit{3.})$ As $X$ is limited, exists $r > 0$ such that $X \subseteq \B(\N, r)$. So, if $A \in X$ e $M = r + 1 > 0$, then $\ o(A) = |o(A) - 1| + 1 \leq d(A, \N) + 1 < r + 1 = M$.

$(\textit{3.} \Rightarrow \textit{1.})$ If exists $M > 0$ such that $o(A) \leq M, \ \forall \ A \in X$, then $X \subseteq \dibigcup_{h = 1}^{\lfloor M \rfloor} \A_h$, in other words, $X$ is contained in a totally bounded set. Then $X$ is totally bounded. \end{proof}

\begin{cor}

$\A$ is not connected, neither complete nor compact.

\end{cor}

\begin{proof} $\A$ is not connected, because we can write $\A = \A_1 \cup \left(\dibigcup_{h = 2}^{\infty}\A_h \right)$. $\A$ is not complete too: if this were the case, $\A_h$, being a closed subset of $\A$, would be complete too. At least, as $\A$ is not complete, $\A$ is not compact. \end{proof}

The lemma below gives us another example of open and closed set in $\A$.

\begin{lema} \label{prop8deu3}

If $\Gamma \subseteq \N$ is finite, then $X = \{A \in \A; \ \Gamma \subseteq A\}$ is open and closed in $\A$.

\end{lema}

\begin{proof} The result is trivial in $\Gamma \subseteq \{0, 1\}$. So, suppose $\Gamma \nsubseteq \{0, 1\}$ and let $x = \max \Gamma \geq 2$. To show that $X$ is open, let $A \in X$. I claim that $\B\left(A, \difrac{1}{x^2}\right) \subseteq X$. In fact, if $B \in \B\left(A, \difrac{1}{x^2}\right)$, then $d(A, B) < \difrac{1}{x^2} \Rightarrow A \cap \overleftarrow{x} = B \cap \overleftarrow{x} \Rightarrow \Gamma \subseteq B \Rightarrow B \in X$.

To show that $X$ is closed, let $\{A_n\}_{n \in \N} \subseteq X$ such that $A_n \rightarrow A \in \A$. We need to show that $A \in X$. As $A_n \rightarrow A$, given $\varepsilon = \difrac{1}{x^2} > 0$ exists $n_0 \in \N$ such that if $n \geq n_0$, then $d(A, A_n) < \difrac{1}{x^2}$. In particular, $d(A, A_{n_0}) < \difrac{1}{x^2} \Rightarrow A \cap \overleftarrow{x} = A_{n_0} \cap \overleftarrow{x}$. As $A_{n_0} \in X$, then $\Gamma \subseteq A_{n_0} \Rightarrow \Gamma \subseteq A \Rightarrow A \in X$. \end{proof}

With the lemma \ref{prop8deu3} we can be more specific with regard to the question of how $\A$ is disconnected. For this, we recall the following definition:

\begin{defi} \index{espaço topológico!totalmente separado}

A topological space $(T, \tau)$ is totally separated if given two distinct points $u, \ v \in T$ exists two open and disjoint sets $U, \ V \in \tau$ such that $u \in U, \ v \in V$ and $T = U \cup V$.

\end{defi}

\begin{prop}

$\A$ is totally separated. In particular, $\A$ is totally disconnected.

\end{prop}

\begin{proof} Let $A_1, \ A_2 \in \A$, with $A_1 \neq A_2$. So $A_1 \triangle A_2 \neq \emptyset$. Let $x = \min (A_1 \triangle A_2)$ and suppose without loss of generality that $x \in A_1$. Write $$\A = \{A \in \A; \ x \in A\} \cup \{A \in \A; \ x \not \in A\}.$$ So, $A_1 \in \{A \in \A; \ x \in A\}, \ A_2 \in \{A \not \in \A; \ x \in A\}$ and both $\{A \in \A; \ x \in A\}$ as this $\{A \in \A; \ x \not \in A\}$ are open subsets of $\A$, by lemma \ref{prop8deu3}. So $\A$ is totally separated. As any totally separated topological space is totally disconnected, the corollary is proved. \end{proof}

Now we define a concept that will be very useful for our study.

\begin{defi} \index{primeiro pior caso}

If $A \in \A$, we define the first worst case $A$ por $$\fwc(A) = \displaystyle \min\{n \in \N; \ o_A(n) = o(A)\}.$$

\end{defi}

\begin{exem}

$\fwc(\N^2) = 7, \fwc(\N^3) = 23$ and $\fwc(A) = 0$ if, and only if, $A = \N$.

\end{exem}

Note that does not exists $A \in \A$ such that $\fwc(A) = 1$.

\begin{lema}

Let $A \in \A$. If $B \supseteq A$ and $o(A) = o(B)$, then $\fwc(B) \geq \fwc(A)$.

\end{lema}

\begin{proof} Suppose by contradiction that $\fwc(B) < \fwc(A)$. Then $o(A) = o(B) = o_B(\fwc(B)) \leq o_A(\fwc(B)) < o(A)$, contradiction. \end{proof}

\begin{prop}

The only isolated point of $\A$ is $\N$.

\end{prop}

\begin{proof} $\N$ is an isolated point of $\A$ because $B(\N, 1) = \{\N\}$. In fact, if $d(A, \N) < 1$, then $o(A) = o(\N) = 1 \Rightarrow A = \N$. Now, let $A \in \A, \ A \neq \N$. Then $o(A) \geq 2$. Let $\varepsilon > 0$. We will show that $A$ ia an accumulation point of $\A$.

(Mudar $A^C$ para $A^c$.) If $A^c$ is infinite, let $n_0 \in \N$ such that $\disum_{\alpha > n_0} \difrac{1}{\alpha^2} < \varepsilon$ and let $n_1 = \max\{n_0, \fwc(A)\}$. As $A^C$ is infinite, exists $x \in \A^C$ such that $x > n_1$. So $B = A \cup \{x\}$ is such that $B \neq A, \ o(B) = o(A)$ (because $x > n_1 \geq \fwc(A)$) and $d(A, B) = \difrac{1}{x^2} < \disum_{\alpha > n_1} \difrac{1}{\alpha^2} \leq \disum_{\alpha > n_0} \difrac{1}{\alpha^2} < \varepsilon$.

If $A^C$ is finite, let $n_0 \in \N$ such that $\disum_{\alpha > n_0} \difrac{1}{\alpha^2} < \varepsilon$ and let $n_1 \in \N$ such that if $n \geq n_1$, then $n \in A$. If $x = \max\{n_0, n_1, \fwc(A)\}$, then $B = A - \{x + 1\}$ is such that $B \neq A, \ o(B) = o(A)$ (because $x \in A \Rightarrow o_B(x + 1) = 2 \leq o(A)$) and $d(A, B) = \difrac{1}{(x + 1)^2} < \disum_{\alpha > x} \difrac{1}{\alpha^2} \leq \disum_{\alpha > n_0} \difrac{1}{\alpha^2} < \varepsilon$. \end{proof}

\begin{prop}

The function $\begin{array}[t]{cccc} \fwc: & \A & \rightarrow & \N \\ & A & \mapsto & \fwc(A) \end{array}$ is continuous in $\A$, but is not uniformly continuous in $\A$.

\end{prop}

\begin{proof} Let $A \in \A, \ A \neq \N$ to avoid the trivial case. Given $\varepsilon > 0$, let ${\delta = \difrac{1}{\fwc(A)^2} > 0}$. So, if $d(A, B) < \delta \leq 1$, then $o(A) = o(B)$ and $A \cap \overleftarrow{\fwc(A)} = B \cap \overleftarrow{\fwc(A)}$. This implies $\fwc(A) = \fwc(B) \Rightarrow |\fwc(A) - \fwc(B)| = 0 < \varepsilon$. As $A$ is any additive basis, $\fwc$ is continuous in $\A$.

If $\fwc$ were uniformly continuous on $\A$, given $\varepsilon = \difrac{1}{2}$ it would exists $\delta > 0$ such that if $d(A, B) < \delta$, then $|\fwc(A) - \fwc(B)| < \difrac{1}{2}, \ \forall \ A, \ B \in \A$. In particular, this implies that $\fwc(A) = \fwc(B), \ \forall \ A, \ B \in \A$ such that $d(A, B) < \delta$. Let $x \in \N - \{0, 1\}$ such that $\difrac{1}{x^2} + \difrac{1}{(x + 1)^2} < \delta$. If $A = \N - \{x\}$ and $B = \N - \{x + 1\}$, then $o(A) = o(B) = 2, \ d(A, B) = \difrac{1}{x^2} + \difrac{1}{(x + 1)^2} < \delta$ and $\fwc(A) = x \neq x + 1 = \fwc(B)$, contradiction. \end{proof}

\begin{cor} \label{cordoppclimitado}

Let $X \subseteq \A$. If exists $M > 0$ such that $\fwc(A) \leq M, \ \forall \ A \in X$, then $\fwc(A) \leq M, \ \forall \ A \in \overline{X}$.

\end{cor}

\begin{proof} Let $A \in \overline{X}$. Then exists $\{A_n\}_{n \in \N} \subseteq X$ such that $A_n \rightarrow A$. As the function $\fwc$ is continuous in $\A, \ \fwc(A_n) \rightarrow \fwc(A)$. As all points of $\N$ are isolated, exists $n_0 \in \N$ such that $\fwc(A_n) = \fwc(A), \ \forall \ n \geq n_0$. So, $\fwc(A) = \fwc(A_{n_0}) \leq M$. \end{proof}

To continue our study of basic properties of the metric space $(\A, d)$ we can now define the concept of pseudoconvergence for sequences in $\A$.

\begin{defi} \index{pseudoconvergência}

Let $\{A_n\}_{n \in \N} \subseteq \A$. We say that $\{A_n\}_{n \in \N}$ pseudoconverges to $A \in \A$ if given $\varepsilon > 0$ exists $n_0 \in \N$ such that if $n \geq n_0$, then $\soma{A_n}{A} < \varepsilon$. We will denote this fact by $A_n \pseudo A$.

\end{defi}

It follows directly from the lemma \ref{dibigcapdibigcup} that if $\{A_n\}_{n \in \N} \subseteq \A$ is such that $A_n \pseudo A \in \A$, then $A = \dibigcup_{j = 0}^{\infty} \dibigcap_{m = j}^{\infty} A_m$. In particular, the limit of a pseudoconvergence, when exists, is unique.

\begin{prop}

Let $\{A_n\}_{n \in \N} \subseteq \A$. If $A_n \rightarrow A \in \A$, then $A_n \pseudo A$. In particular, $A = \dibigcup_{j = 1}^{\infty} \dibigcap_{m = j}^{\infty} A_m$.

\end{prop}

\begin{proof} As $A_n \rightarrow A$, dado $\varepsilon > 0$ exists $n_0 \in \N$ such that $n \geq n_0$, then $d(A, A_n) < \varepsilon$. So, if $n \geq n_0, \ \soma{A_n}{A} \leq d(A_n, A) < \varepsilon$. Then $A_n \pseudo A$. \end{proof}

The reciprocal of the previous proposition is not true. In fact, consider the sequence of additive basis $A_n = \N - \{n\}, \ \forall \ n \in \N - \{0, 1\}$. Of course $\{A_n\}_{n \in \N - \{0, 1\}}$ diverges. However, $A_n \pseudo \N$. In fact, given $\varepsilon > 0$, let $n_0 \in \N - \{0, 1\}$ such that $\difrac{1}{\fantasma{n}{0}{2}} < \varepsilon$. So, if $n \geq n_0, \ \soma{A_n}{\N} = \difrac{1}{n^2} \leq \difrac{1}{\fantasma{n}{0}{2}} < \varepsilon \Rightarrow A_n \pseudo \N$.

\begin{lema} \label{pseudoconvergeordemmaiorigual}

Let $\{A_n\}_{n \in \N} \subseteq \A$. If $A_n \pseudo A \in \A$, then exists $n_0 \in \N$ such that if $n \geq n_0$, then $o(A_n) \geq o(A)$.

\end{lema}

\begin{proof} Let $x = \fwc(A) + 1$. As $A_n \pseudo A$, given $\varepsilon = \difrac{1}{x^2}$ exists $n_0 \in \N$ such that if $n \geq n_0$, then $\soma{A_n}{A} < \difrac{1}{x^2}$. In particular, $A_n \cap \overleftarrow{x} = A \cap \overleftarrow{x}, \ \forall \ n \geq n_0$. So, if $n \geq n_0$, then $o_{A_n}(\fwc(A)) = o_A(\fwc(A)) = o(A)\Rightarrow o(A_n) \geq o(A), \ \forall \ n \geq n_0$. \end{proof}

\begin{teo} \label{proppseudomaisppclimitado}

Let $\{A_n\}_{n \in \N} \subseteq \A$. In $A_n \pseudo A \in \A$ and if exists $M > 0$ such that $\fwc(A_n) \leq M, \ \forall \ n \in \N$, then $A_n \rightarrow A$.

\end{teo}

\begin{proof} I claim that it is sufficient to show that there is $n_1 \in \N$ such that if $n \geq n_1$, then $o(A) = o(A_n)$. In fact, with this, given $A_n \pseudo A$, then given $\varepsilon > 0$ exists $n_0 \in \N$ such that $n \geq n_0$, then $\soma{A_n}{A} < \varepsilon$. If $n_2 = \max\{n_0, n_1\}$ and if $n \geq n_2$, then $d(A_n, A) = \soma{A_n}{A} < \varepsilon \Rightarrow A_n \rightarrow A$.

By the previous lemma, exists $n_3 \in \N$ such that if $n \geq n_3$, then $o(A_n) \geq o(A)$. Suppose by contradiction that does not exists $n_1 \in \N$ such that if $n \geq n_1$, then $o(A) = o(A_n)$. Thereby there is a subsequence $\{A_m\}_{m \in \N} \subseteq \{A_n\}_{n \in \N}$ such that $o(A_m) > o(A), \ \forall \ m \in \N$.

So, let $y \in \N^*$. As $A_n \pseudo A$, given $\varepsilon = \difrac{1}{y^2}$ exists $n_4 \in \N$ such that if $n \geq n_4$, then $\soma{A_n}{A} < \difrac{1}{y^2}$. In particular, $A_n \cap \overleftarrow{y} = A \cap \overleftarrow{y}, \ \forall \ n \geq n_4$. If $m \in \N$ is such that $m \geq n_4$, then $A \cap \overleftarrow{y} = A_m \cap \overleftarrow{y}$. If $z \in \N$, $z \leq y$, then $o_{A_m}(z) = o_A(z) \leq o(A)$. As $o(A_m) > o(A)$, then $\fwc(A_m) > y$. As $y$ is arbitrarily, does not exists $M > 0$ such that $\fwc(A_n) \leq M, \ \forall \ n \in \N$. Contradiction. \end{proof}

\begin{teo} \label{teosequenciadecauchy}

Every Cauchy sequence in $\A$ pseudoconverges.

\end{teo}

\begin{proof} Let $\{A_n\}_{n \in \N} \subseteq \A$ a Cauchy sequence and let $A = \dibigcup_{m = 0}^{\infty} \dibigcap_{n = m}^{\infty} A_n$. We will show that $A \in \A$ and that $A_n \pseudo A$.

As $\{A_n\}_{n \in \N}$ is a Cauchy sequence, given $\varepsilon = \difrac{1}{2}$ exists $n_0 \in \N$ such that if $n, \ m \geq n_0$, then $d(A_n, A_m) < \difrac{1}{2}$. In particular, $o(A_n)= o(A_m), \ \forall \ n, \ m \geq n_0$. So exists $\dilim_{n \rightarrow \infty} o(A_n)$. Let $L = \dilim_{n \rightarrow \infty} o(A_n)$. To show that $A \in \A$ is sufficient to show that $o_A(x) \leq L, \ \forall \ x \in \N$. For this, let $x \in \N, \ x \geq 2$, because $\{0, 1\} \subseteq A$.

As $\{A_n\}_{n \in \N}$ is a Cauchy sequence, given $\varepsilon = \difrac{1}{x^2}$ exists $n_1 \in \N$ such that $n, \ m \geq n_1$, then $d(A_n, A_m) < \difrac{1}{x^2}$. In particular, $A_n \cap \overleftarrow{x} = A_m \cap \overleftarrow{x}, \ \forall \ n, \ m \geq n_1$. Let $n_2 = \max\{n_0, n_1\}$. I claim that $A_{n_2} \cap \overleftarrow{x} = A \cap \overleftarrow{x}$. In fact:

$(\subseteq)$ Let $y \in A_{n_2}, \ y \leq x$. We know that $\dibigcap_{m = n_2}^{\infty} A_m \subseteq A$. As $n_2 \geq n_1$, then $$A_{n_2} \cap \overleftarrow{x} = A_m \cap \overleftarrow{x}, \ \forall \ m \geq n_2.$$ In particular, $y \in A_m, \ \forall \ m \geq n_2 \Rightarrow y \in A$.

$(\supseteq)$ Let $y \in A, \ y \leq x$. As $y \in A$, exists $m \in \N$ such that $y \in \dibigcap_{n = m}^{\infty} A_n$. Let $n_3 = \max\{n_2, m\}$. As $n_3 \geq m$, then $y \in A_{n_3}$. As $n_3 \geq n_2$ and $y \leq x$, then $y \in A_{n_2}$.

So $o_A(x) = o_{A_{n_2}}(x) \leq o(A_{n_2}) = L$.

We will show now that $A_n \pseudo A$. For this, let $\varepsilon > 0$. Let $z \in \N^*$ such that $\disum_{\alpha > z} \difrac{1}{\alpha^2} < \varepsilon$. As $\{A_n\}_{n \in \N}$ is a Cauchy sequence, given $\varepsilon' = \difrac{1}{z^2}$ exists $n_4 \in \N$ such that if $n, \ m \geq n_4$, then $d(A_n, A_m) < \difrac{1}{z^2}$. In particular, $$A_n \cap \overleftarrow{z} = A_m \cap \overleftarrow{z}, \ \forall \ n, \ m \geq n_4.$$ In the same way that before, $A_{n_4} \cap \overleftarrow{z} = A \cap \overleftarrow{z}$. This implies $A \cap \overleftarrow{z} = A_n \cap \overleftarrow{z}, \ \forall \ n \geq n_4$. So, if $n \geq n_4$, then $\soma{A_n}{A} \leq \disum_{\alpha > z} \difrac{1}{\alpha^2} < \varepsilon$. So $A_n \pseudo A$. \end{proof}

The reciprocal of this theorem does not hold. For example, the sequence $\{A_n\}_{n \in N^*} \subseteq \A$ given by $$A_n = \{0, 1, \ldots, n, n^2 + 1, \ldots\}$$ is such that $A_n \pseudo \N$ but $\{A_n\}_{n \in N^*}$ is not a Cauchy sequence, because $o(A_n) = n, \ \forall \ n \in \N - \{0, 1\}$.

The lemma below enable us to give a fairly simple characterization of compact subsets of $\A$:

\begin{lema} \label{ppclimitadoehcompleto}

Let $X \subseteq \A$. If exists $M > 0$ such that $\fwc(A) \leq M, \ \forall \ A \in X$, then $\overline{X}$ is complete.

\end{lema}

\begin{proof} Let $\{A_n\}_{n \in \N} \subseteq \overline{X}$ a Cauchy sequence. By theorem \ref{teosequenciadecauchy}, exists $A \in \A$ such that $A_n \pseudo A$. As exists $M > 0$ such that $\fwc(A) \leq M, \ \forall \ A \in \A$, the corolary \ref{cordoppclimitado} implies that $\fwc(A_n) \leq M, \ \forall \ n \in \N$. By theorem \ref{proppseudomaisppclimitado}, $A_n \rightarrow A$. As $\overline{X}$ is closed, $A \in \overline{X}$. Then $\overline{X}$ is complete. \end{proof}

The reciprocal of the previous lemma does not hold. In fact, consider the sequence $A_n = \N - \{2, \ldots, n\}$ and let $X = \{A_n\}_{n \in \N - \{0, 1\}}.$ So $o(A_n) = n = \fwc(A_n)$. Then $X$ is complete (because $d(A_n, A_m) > 1$ if $n \neq m$), but does not exists $M > 0$ such that $\fwc(A) \leq M, \ \forall \ A \in X$.

Also note that we can not conclude that $X$ is complete. In fact, let $A_n = \N - \{2, n\}, \ n \in \N - \{0, 1\}$, and let $X = \{A_n\}_{n \in \N - \{0, 1\}}$. Then $A_n \rightarrow \N - \{2\} \not \in X$. Of course, what happens here is that the additive basis $A$, limit of the sequence $\{A_n\}_{n \in \N - \{0, 1\}}$, can not belong to $X$.

\begin{teo} \label{caracterizacaodoscompactos}

Let $X \subseteq \A$. The following affirmations are equivalent:

\begin{enumerate}
  \item $X$ is compact.
  \item $X$ is closed, limited and exists $M > 0$ such that $\fwc(A) \leq M, \ \forall \ A \in X$.
\end{enumerate}

\end{teo}

\begin{proof} $(\textit{1.} \Rightarrow \textit{2.})$ If $X$ is compact, clearly $X$ must be closed and limited. As the function $\fwc$ is continuous in $\A$, the set $\fwc(X) = \{\fwc(A); \ A \in X\}$ is compact in $\N$. In particular, exists $M > 0$ such that $\fwc(A) \leq M, \ \forall \ A \in X$.

$(\textit{2.} \Rightarrow \textit{1.})$ As $X$ is limited, the corollary \ref{cortotalmentelimitado} implies that $X$ is totally bounded. As exists $M > 0$ such that $\fwc(A) \leq M, \ \forall \ A \in X$, follows from lemma \ref{ppclimitadoehcompleto} that $\overline{X}$ is complete. But then $X = \overline{X} \Rightarrow X$ is complete. So $X$ is compact. \end{proof}

\begin{cor} \label{corolariobolacompacta}

Let $A \in \A, \ A \neq \N$. If $0 < r \leq \difrac{1}{\fwc(A)^2}$, then $\B[A, r]$ is a compact subset of $\A$. In particular, $\A$ is a locally compact metric space.

\end{cor}

\begin{proof} It is sufficient to show that $\B\left[A, \difrac{1}{\fwc(A)^2}\right]$ is a compact subset of $\A$. As $\B\left[A, \difrac{1}{\fwc(A)^2}\right]$ is a bounded and closed subset of $\A$, by the previous theorem it suffices to show that there is $M > 0$ such that $\fwc(A) \leq M, \ \forall \ A \in \B\left[A, \difrac{1}{\fwc(A)^2}\right]$. By corollary \ref{cordoppclimitado}, it is sufficient to show that exists $M > 0$ such that $\fwc(A) \leq M, \ \forall \ A \in \B\left(A, \difrac{1}{\fwc(A)^2}\right)$. For this, let $B \in \B\left(A, \difrac{1}{\fwc(A)^2}\right)$. So $d(A, B) < \difrac{1}{\fwc(A)^2} \leq \difrac{1}{4}$. In particular, $o(A) = o(B)$. Again, as $d(A, B) < \difrac{1}{\fwc(A)^2}$, then $A \cap \overleftarrow{\fwc(A)} = B \cap \overleftarrow{\fwc(A)} \Rightarrow \fwc(A) = \fwc(B)$. \end{proof}

Note that, in the space $(\A, d)$, small balls are compact!

\begin{cor}

Let $\A_h^m = \{A \in \A; \ o(A) = h \text{ e } \fwc(A) = m\}$. Then each $\A_h^m$ is a compact open subset of $\A, \ \A = \dibigcup_{{h \in \N^*} \atop {m \in \N}} \A_h^m$ and the $\A_h^m$ sets are pairwise disjoint.

\end{cor}

\begin{proof} We know that the functions $o: \ \A \rightarrow \N^*$ and $\fwc: \ \A \rightarrow \N$ are continuous. As $\A_h^m = o^{-1}(\{h\}) \cap \fwc^{-1}(\{m\})$, follows that each $\A_h^m$ is open in $\A$. The compactness follows directly from theorem \ref{caracterizacaodoscompactos} and the other affirmations are obvious. \end{proof}

Note that some open balls can be compact, requiring for it, just a small condition. Let $A \in \A, \ A \neq \N$. If $0 < r \leq \difrac{1}{\fwc(A)^2}$ and $\mathcal{S}(A, r) = \emptyset$, then $\B(A, r)$ is a compact subset of $\A$. To see this, just remember that $\B[A, r] = \B(A, r) \cup \mathcal{S}(A, r)$.

\begin{defi} \index{conjunto magro}

Let $(T, \tau)$ a topological space. We say that $X \subseteq T$ is meagre in $T$ if $X = \dibigcup_{n \in \N} X_n$ is such that $\inte(\overline{X_n}) = \emptyset, \ \forall \ n \in \N$.

\end{defi}

\begin{defi} \index{espaço topológico!topologicamente completo}

A topological space $(T, \tau)$ is topologically complete is it is homeomorphic to a complete metric space.

\end{defi}

As $\A$ is locally compact, $\A$ is topologically complete. So the Baire theorem holds in $\A$:

\begin{teo} [Baire] \index{teorema!de Baire}

Every meagre set in $\A$ has empty interior. Equivalently: if $X = \dibigcup_{n \in \N} X_n$, where each $X_n$ is closed in $\A$ and has empty interior, then $\inte(X) = \emptyset$. Yet: every countable intersection os open dense subsets is a dense subset of $\A$.

\end{teo}

The following proposition terminating the study of the basic properties of the metric space $(\A, d)$:

\begin{prop}

The set $\{A \in \A; \ A^C \textit{ é finite}\}$ is countable and dense in $\A$. In particular, $\A$ is a separable metric space.

\end{prop}

\begin{proof} Let $X = \{A \in \A; \ A^C \textit{ é finito}\}$. Of course, $X$ is countable. We will show that $\overline{X} = \A$. For this, let $A \in \A$ e $\varepsilon > 0$, and let $n_0 \in \N$ such that $\disum_{\alpha > n_0} \difrac{1}{\alpha^2} < \varepsilon$. Let $x = \max \{n_0, \fwc(A)\} + 1$. If $B = A \cup \{x, x + 1, x + 2, \ldots\}$, then $B \in X, \ o(B) = o(A)$ (because $x > \fwc(A)$), and $d(A, B) = \soma{A}{B} \leq \disum_{\alpha \geq x}^{\infty} \difrac{1}{\alpha^2} \leq \disum_{\alpha > n_0}^{\infty} \difrac{1}{\alpha^2} < \varepsilon$. So $\overline{X} = \A$. \end{proof}

\section{Functions from $\A$ to $\A$}

Add elements in an additive basis is a difficult process to understand. In fact, we can add a single number in a given base additive to decrease its order (put the number $2$ in $\N - \{2\}$), or we can add an infinite set of numbers in an additive basis and not change their order (put \{5, 7, 9, 11, 13, 15, \ldots\} in $\{0, 1, 2, 4, 6, 8, 10, \ldots\}$).

In this section we study in more depth way some functions $ f: \A \rightarrow \A$ which have certain properties. At the end of this section, we present a theorem that allows us (not in a direct way) see if we can add certain set of numbers in a base without changing their order.

If $f: \A \rightarrow \A$ is a function, let's denote $\mathcal{C}(f) = \{A \in \A; \ f \textit{ is continuous in } \A\}$.

\begin{defi} \index{função!ordenativa}

A function $f: \A \rightarrow \A$ is called ordenative in $A \in \A$ is exists $r > 0$ such that if $B \in \B(A, r)$, then $o(f(A)) = o(f(B))$. A function $f: \A \rightarrow \A$ is called ordenative in $X \subseteq \A$ if $f$ is ordenative in $A, \ \forall \ A \in X$.

\end{defi}

Another type of function to be studied are the pseudocontinuous functions.

\begin{defi} \index{função!pseudocontínua}

A function $f: \A \rightarrow \A$ is pseudocontinuous in $A \in \A$ if $A_n \rightarrow A$ implies $f(A_n) \pseudo f(A)$. A function $f: \A \rightarrow \A$ is pseudocontinuous in $X \subseteq \A$ if $f$ is pseudocontinuous in $A, \ \forall \ A \in X$.

\end{defi}

\begin{exem} If $\Gamma \subseteq \N$ and $\begin{array}[t]{cccc} f_\Gamma: & \A & \rightarrow & \A \\ & A & \mapsto & A \cup \Gamma \end{array}$, then $f_\Gamma$ is a pseudocontinuous function in $\A$.

\end{exem}

\begin{exem} Se $\{1\} \subseteq \Gamma \subseteq \N^*$ e $\begin{array}[t]{cccc} g_\Gamma: & \A & \rightarrow & \A \\ & A & \mapsto & \dibigcup_{i \in \Gamma} A^i \end{array}$, então $g_\Gamma$ é uma função pseudocontínua em $\A$.

\end{exem}

\begin{teo} \label{continuasseordenativapseudo}

let $A \in \A$. The following statements are equivalent with respect to a function $f: \A \rightarrow \A$:

\begin{enumerate}
  \item $f$ is continuous in $A$.
  \item $f$ is ordenative and pseudocontinuous in $A$.
\end{enumerate}

\end{teo}

\begin{proof} $(\textit{1.} \Rightarrow \textit{2.})$ Suppose $f$ is continuous in $A$. So given $\varepsilon = \difrac{1}{2}$ exists $\delta > 0$ such that if $d(A, B) < \delta$, then $d(f(A), f(B)) < \difrac{1}{2}$. In particular, if $d(A, B) < \delta$, then $o(f(A)) = o(f(B))$, in other words, $f$ is ordenative in $A$. To see that $f$ is pseudocontinuous in $A$, let $\{A_n\}_{n \in \N} \subseteq A$ such that $A_n \rightarrow A$. As $f$ is continuous in $A$, $f(A_n) \rightarrow f(A) \Rightarrow f(A_n) \pseudo f(A)$.

$(\textit{2.} \Rightarrow \textit{1.})$ Let $\{A_n\}_{n \in \N} \subseteq A$ such that $A_n \rightarrow A$. We will show that $f(A_n) \rightarrow f(A)$. For this, let $\varepsilon > 0$. As $f$ is pseudocontinuous in $A, \ f(A_n) \pseudo f(A)$, then exists $n_0 \in \N$ such that if $n \geq n_0$, then $\soma{f(A_n)}{f(A)} < \varepsilon$. Now, as $f$ is ordenative in $A$, exists $r > 0$ such that if $d(A, B) < r$, then $o(f(A)) = o(f(B))$. As $A_n \rightarrow A$, exists $n_1 \in \N$ such that if $n \geq n_1$, then $d(A, A_n) < r \Rightarrow o(f(A_n)) = o(f(A))$. Let $n_2 = \max\{n_0, n_1\}$. So, if $n \geq n_2$, then $$d(f(A_n), f(A)) = \dd{f(A_n)}{f(A)} = \soma{f(A_n)}{f(A)} < \varepsilon.$$

Therefore $f(A_n) \rightarrow f(A)$. \end{proof}

\begin{exem}

The function $\begin{array}[t]{cccc} f: & \A & \rightarrow & \A \\ & A & \mapsto & \left\{ \begin{array}{cl} A & \text{, se } A \neq \N - \{2\} \\ \N - \{3\} & \text{, se } A = \N - \{2\} \end{array}\right.\end{array}$ is ordenative in $\A$ (because $o(A) = o(f(A)), \ \forall \ A \in \A$), but is not continuous in $\N - \{2\}$.

\end{exem}

\begin{exem} 

The function $\begin{array}[t]{cccc} f: & \A & \rightarrow & \A \\ & A & \mapsto & \left\{ \begin{array}{cl} A & \text{, se } 2 \in A \\ A \cup \{2\} & \text{, se } 2 \not \in A \end{array}\right.\end{array}$ is pseudocontinuous in $\A$ but is not continuous in $\N - \{2\}$.

\end{exem}

\begin{prop}

If $f: \A \rightarrow \A$ is a pseudocontinuous function in $\A$, then

\begin{enumerate}
  \item $\mathcal{C}(f)$ is an open subset of $\A$.
  \item if $K$ is a compact subset of $\A$, then $f(K)$ is closed in $\A$.
\end{enumerate}

\end{prop}

\begin{proof} Let $A \in \mathcal{C}(f)$. Given $\varepsilon = \difrac{1}{2}$, exists $\delta > 0$ such that if $d(A, B) < \delta$, then $d(f(A), f(B)) < \difrac{1}{2}$. In particular, $o(f(A)) = o(f(B)), \ \forall \ B \in \B(A, \delta)$. I claim that $\B(A, \delta) \subseteq \mathcal{C}(f)$.

in fact, let $B \in \B(A, \delta)$. As $f$ is pseudocontinuous in $B$, by theorem \ref{continuasseordenativapseudo} it is sufficient to show that $f$ is ordenative in $B$. For this, let $r > 0$ such that $\B(B, r) \subseteq \B(A, \delta)$. If $C \in \B(B, r)$, then $C \in \B(A, \delta) \Rightarrow o(f(C)) = o(f(A)) = o(f(B))$. As $C$ is arbitrarily, $f$ is ordenative in $B$ and, therefore, continuous in $B$.

Now, suppose $K$ is a compact and let $\{f(A_n)\}_{n \in \N}$ be a sequence in $f(K)$ such that $f(A_n) \rightarrow B \in \A$. We need to show that $B \in f(K)$. As $K$ is a compact, the sequence $\{A_n\}_{n \in \N}$ admits a convergent subsequence for, let's say, $B' \in K$. Suppose without loss of generality that the sequence itself converges to $B'$. As $A_n \rightarrow B'$ and $f$ é pseudocontinuous, then $f(A_n) \pseudo f(B')$. As $f(A_n) \rightarrow B$, then $f(A_n) \pseudo B$. By uniqueness of the limit, $B = f(B') \Rightarrow B \in f(K)$. \end{proof}

\begin{defi}

A function $f: \A \rightarrow \A$ is called expansive if $f(A) \supseteq A, \ \forall \ A \in \A$.

\end{defi}

\begin{teo}

Let $f: \A \rightarrow \A$ an expansive function. If $f$ is  pseudocontinuous in $A \in \A$ and $o(A) = o(f(A))$, then $f$ is continuous in $A$.

\end{teo}

\begin{proof} Let $A_n \rightarrow A$ and $\varepsilon > 0$. Suppose $A \neq \N$ to avoid the trivial case.

as $A_n \rightarrow A$, then for the number $\difrac{1}{2} > 0$ exists $n_0 \in \N$ such that if $n \geq n_0$, then $d(A, A_n) < \difrac{1}{2} \Rightarrow o(A) = o(A_n), \ \forall \ n \geq n_0$. As $A_n \subseteq f(A_n), \ \forall \ n \in \N$, then $o(f(A_n)) \leq o(A_n) = o(A) = o(f(A)), \ \forall \ n \geq n_0$.

Now, as $f$ is pseudocontinuous in $A$, then $f(A_n) \pseudo f(A)$. Note that $A \neq \N$ and $o(A) = o(f(A)) \Rightarrow f(A) \neq \N \Rightarrow \fwc(f(A)) > 0$. So, for the number $\difrac{1}{\fwc(f(A))^2} > 0$ exists $n_1 \in \N$ such that if $n \geq n_1$, then $\soma{f(A_n)}{f(A)} < \difrac{1}{\fwc(f(A))^2}.$

So $f(A_n) \cap \overleftarrow{\fwc(f(A))} = f(A) \cap \overleftarrow{\fwc(f(A))}, \ \forall \ n \geq n_1$. Therefore $o(f(A_n)) \geq o_{f(A_n)}(\fwc(f(A))) = o_{f(A)}(\fwc(f(A))) = o(f(A)), \ \forall \ n \geq n_1$.

Moreover, for that $\varepsilon > 0$, exists $n_2 \in \N$ such that if $n \geq n_2$, then $\soma{f(A_n)}{f(A)} < \varepsilon$.

So, let $n_3 = \max\{n_0, n_1, n_2\}$. Thereby, if $n \geq n_3$, then $d(f(A_n), f(A)) = |o(f(A_n)) - o(f(A))| + \soma{f(A_n)}{f(A)} = \soma{f(A_n)}{f(A)} < \varepsilon \Rightarrow f(A_n) \rightarrow f(A) \Rightarrow f$ is continuous in $A$. \end{proof}

\begin{teo} [More general]

Let $f: \A \rightarrow \A$ and suppose that $f$ is continuous in $A \in \A$. Suppose that exists $r > 0$ such that exists $D \subseteq \B(A, r)$ dense in $\B(A, r)$ such that $o(B) = o(f(B)), \ \forall \ B \in D$. Then $o(A) = o(f(A))$.

\end{teo}

\begin{proof} As $f$ is continuous in $A$, given $\varepsilon = \difrac{1}{2}$, exists $\delta > 0$ such that if $d(A, B) < \delta$, then $d(f(A), f(B)) < \varepsilon \Rightarrow o(f(A)) = o(f(B))$. Suppose without loss of generality that $\delta < \min\{r, 1\}$. As $D$ is dense in $\B(A, r)$, exists $B \in D \cap \B(A, \delta)$. So $o(A) = o(B) = o(f(B)) = o(f(A))$. \end{proof}

\section{Thin bases}

\begin{lema}

Let $A \in \A, \ h \in \N^*$ and $c > 0$ satisfying $A(n) \geq c\sqrt[h]{n}, \ \forall \ n \in \N^*$. Then $o(A) \geq h$.

\end{lema}

\begin{proof} The result is trivial if $h = 1$. Suppose $h \geq 2$ and suppose by absurd that $o(A) \geq h - 1$. Then $(h - 1) \cdot A = \N$. If $n \in \N^*$, then $$n = ((h - 1)A)(n) \leq [1 + A(n)]^{h - 1} \leq [2A(n)]^{h - 1} \leq (2c\sqrt[h]{n})^{h - 1} = (2c)^{h - 1}n^{\frac{h - 1}{h}}.$$ Dividing by $n$, we obtain $$1 \leq (2c)^{h - 1}n^{-\frac{1}{h}} \rightarrow 0,$$ absurd. \end{proof}

\begin{defi}

A basis $A \in \A$ is called thin if exists $c > 0$ such that $$A(n) \leq cn^{\frac{1}{o(A)}}, \ \forall \ n \in \N^*.$$

\end{defi}

Clearly $\N$ is a thin basis. Nathanson, in his article Cassels Bases, show the construction independently realized by Raikov and Stöhr of thin bases of order $h, \ \forall \ h \geq 2$. The Raikov-Störh bases will be the central objects of this section.

\begin{lema}

Let $A_1, \ldots, A_m$ thin bases, all with order $h$. So $B = \dibigcup_{i = 1}^{m} A_i$ is a thin bases of order $h$.

\end{lema}

\begin{proof} By hypothesis, exists constants $c_1, \ldots, c_m > 0$ such that $$A_i(n) \leq c_in^{\frac{1}{h}}, \ \forall \ n \in \N^*, \ \forall \ i \in \{1, \ldots, m\}.$$ Let $n \in \N^*$. Then $$B(n) = \left(\dibigcup_{i = 1}^{m} A_i\right) (n) \leq \disum_{i = 1}^{m} A_i(n) \leq \disum_{i = 1}^{m} c_in^{\frac{1}{h}} = \left(\disum_{i = 1}^{m} c_i\right) n^{\frac{1}{h}}.$$ By the previous lemma, $o(B) \geq h$. Now, as $B \supseteq A_1$, then $o(B) \leq o(A_1) = h$. So $o(B) = h$ and this lemma is proved. \end{proof}

\begin{defi}

Let $\{0, 1\} \subseteq A \subseteq \N$ and $m \in \N$. Denote $A + m = \{a + m; \ a \in A\}$.

\end{defi}

\begin{prop}

If $A \in \A$ is a thin bases of order $h$ and $m \in \N$, then $B = \dibigcup_{i = 0}^{m} (A + i)$ is a thin bases of order $h$.

\end{prop}

\begin{proof} Let $n \in \N^*$. Then $$B(n) \left(\dibigcup_{i = 0}^{m} (A + i)\right) (n) \leq \disum_{i = 0}^{m} (A + i)(n) \leq \disum_{i = 0}^{m} A(n) = (m + 1) \cdot A(n) \leq (m + 1)c \cdot n^{\frac{1}{h}}.$$ By lemma 1, $o(B) \geq h$. Now, as $B \supseteq A$, then $o(B) \leq o(A) = h$. So $o(B) = h$. \end{proof}

\begin{prop}

If $A \in \A$ is a thin basis, then $o(A) = o(\mathcal{G}(A))$.

\end{prop}

\begin{proof} By hypothesis, exists $c > 0$ such that $A(n) \leq cn^{\frac{1}{h}}, \ \forall \ n \in \N^*$. Suppose by absurd that $o(A) \neq o(\mathcal{G}(A))$. Then $o(\mathcal{G}(A)) < o(A)$ and $o(A) \geq 2$. Let $h = o(A) - 1 \in \N^*$. Since we are assuming $o(\mathcal{G}(A)) < o(A)$, then $h \cdot \mathcal{G}(A) = \N$. So $$[h \cdot \mathcal{G}(A)](n) = n, \ \forall \ n \in \N^*.$$ Therefore, if $n \in \N^*$, then \begin{eqnarray*} n & = & [1 + \mathcal{G}(A)(n)]^h \leq 2^h\mathcal{G}(A)(n)^h \leq 2^h \left[\difrac{\ln(n)A(n)}{\ln(2)}\right]^h \leq 2^h \left[\difrac{\ln(n) \cdot c n^{\frac{1}{h + 1}}}{\ln(2)}\right]^h = \\ & = & \left(\difrac{2c}{\ln(2)}\right)^h \ln(n)^h n^{\frac{h}{h + 1}} \end{eqnarray*} Dividing by $n$, $$1 \leq \left(\difrac{2c}{\ln(2)}\right) \ln(n)^h n^{\frac{-1}{h + 1}} \rightarrow 0,$$ absurd. So $o(\mathcal{G}(A)) = o(A)$. \end{proof}

\begin{teo}

Let $F$ be the set of all thin bases. Then $\overline{F} = \A$. In particular, the following affirmations are equivalents: \begin{itemize}
                                                                                                   \item $\mathcal{G}$ is continuous in $A \in \A$
                                                                                                   \item $o(A) = o(\mathcal{G}(A))$
                                                                                                 \end{itemize} \end{teo}

\begin{proof} As $E$ is dense in $\A$, we will show that $F$ is dense in $E$. For this, let $A \in E$ e $\varepsilon > 0$, and suppose that $A \neq \N$ to avoid the trivial case. Then $h = o(A) \geq 2$. Write $A = \{0, 1\} \cup \widetilde{A} \cup \overrightarrow{n_0}$, where $n_0$ is the smallest natural number such that $n \geq n_0$, then $n \in  A$. Then $n_0 \geq 3$ and $\widetilde{A} \subseteq \{2, \ldots n_0 - 1\}.$ Let $T_h$ the Raikov-Stöhr basis of order $h$ (so $T_h$ is thin). Let $n_1 \in \N^*$ such that $\disum_{\alpha > n_1} \difrac{1}{\alpha^2} < \varepsilon$. Suppse without loss of generality that $n_1 \geq h_{n_0}$. Let $T = \dibigcup_{i = 0}^{(h - 1)n_0} (T_h + i)$, which, by the previous proposition, is a thin basis of order $h$. Let $$B = \{0, 1\} \cup \widetilde{A} \cup \{n_0, \ldots, n_1\} \cup (T \cap \overrightarrow{n_1 + 1}).$$ I claim that \begin{itemize}
                                                                                                                                   \item $o(B) = o(A)$
                                                                                                                                   \item $d(A, B) < \varepsilon$
                                                                                                                                   \item $B$ is thin
                                                                                                                                 \end{itemize} This ends the theorem.\end{proof}




\begin{thebibliography}{LLL}

\bibitem{Goldbach} Christian Goldbach. [\textit{Carta}] 7 de junho de 1742, Moscou [para] Leonhard Euler, Berlim. 5 f. \textit{Continuation sur les mêmes sujets. Deux théorèmes d'analyse.} Disponível em: \verb"www.math.dartmouth.edu/~euler/correspondence/letters/OO0765.pdf". Acesso em 9 de julho de 2012.

\bibitem{Lima} Elon Lages Lima. \textit{Espaços métricos.} 4. ed. Rio de Janeiro: IMPA, 2007. 299 p. (Projeto Euclides).

\bibitem{Luthar} Indar S. Luthar. \textit{A generalization of a theorem of Landau.} Acta Arithmetica. XII. (1967) p. 223 - 228.

\bibitem{Nlivro} Melvyn B. Nathanson. \textit{Additive number theory: the classical bases.} 1. ed. Estados Unidos da América: Springer-Verlag, 1996. 364 p. (Graduate Texts in Mathematics 164).

\bibitem{Nartigo} \_\_\_\_\_\_. \textit{Cassels bases.} Disponível em: \verb"arxiv.org/pdf/0905.3144v1.pdf". Acesso em 9 de julho de 2012.

\bibitem{Ribenboim} Paulo Ribenboim. \textit{Números primos: mistérios e recordes.} 1. ed. Rio de Janeiro: IMPA, 2001. 292 p. (Coleção Matemática Universitária).

\bibitem{SSG} Salahoddin Shokranian; Marcus Soares; Hemar Godinho. \textit{Teoria dos números.} 2. ed. Brasília: Editora Universidade de Brasília, 1999. 325 p.

\bibitem{Tao} Terence Tao. \textit{Every odd number greater than 1 is the sum of at mosf five primes.} Disponível em: \verb"arxiv.org/pdf/1201.6656v4.pdf". Acesso em 9 de julho de 2012.

\bibitem{Wagstaff} Samuel S. Wagstaff Jr.. \textit{The Schnirelmann density of the sums of three squares.} Proceedings of the American Mathematical Society. Volume 52, October 1975.

\bibitem{Wikipedia} Wikipedia, the free encyclopedia. \textit{Goldbach's conjecture.} Disponível em: \verb"en.wikipedia.org/wiki/Goldbach's_conjecture". Acesso em 9 de julho de 2012.

\end{thebibliography}
\end{document}